\documentclass[11pt]{amsart}

\usepackage{amscd}
\usepackage{amsfonts,amssymb,latexsym} 

\usepackage[all]{xy} 
\xyoption{arc}

\newcommand{\SE}{{\mathcal{E}}}

\newcommand{\SH}{{\mathcal{H}}}
\newcommand{\SI}{{\mathcal{I}}}

\newcommand{\SM}{{\mathcal{M}}}
\newcommand{\SN}{{\mathcal{N}}}
\newcommand{\SO}{{\mathcal{O}}}

\newcommand{\cS}{{\mathcal{S}}}

\newcommand{\SU}{{\mathcal{U}}}

\newcommand{\CC}{\mathbb{C}}

\newcommand{\VV}{\mathbb{V}}

\newcommand{\Spec}{\operatorname{Spec}}

\newcommand{\Pic}{\operatorname{Pic}}
\newcommand{\Jac}{\operatorname{Jac}}

\newcommand{\Sym}{\operatorname{Sym}}
\newcommand{\id}{\operatorname{id}}

\newcommand{\too}{\longrightarrow}

\newcommand{\wt}{\widetilde}
\newcommand{\tr}{\operatorname{tr}}
\newcommand{\GL}{\operatorname{GL}}

\newcommand{\limit}{\operatorname{lim}}
\newcommand{\Ker}{\operatorname{Ker}}

\DeclareMathOperator{\pdeg}{pardeg \,}      
\DeclareMathOperator{\pmu}{par\mu \,}       

\DeclareMathOperator{\PE}{ParEnd\,}
\DeclareMathOperator{\SPE}{SParEnd\,}

\newcommand{\op}{\operatorname}

\newtheorem{proposition}{Proposition}[section]
\newtheorem{theorem}[proposition]{Theorem}

\newtheorem{lemma}[proposition]{Lemma}

\theoremstyle{remark}
\newtheorem{remark}[proposition]{Remark}

\numberwithin{equation}{section}

\title[Torelli theorem for parabolic Higgs bundles]{Torelli
theorem for the moduli space of parabolic Higgs
bundles}

\date{}

\thanks{T.G. supported by the Spanish
Ministerio de Educaci\'on y Ciencia [MTM2007-63582] M.L. supported by a postdoctoral grant at MPIM-Bonn and the grant SFRH/BPD/27039/2006 by F.C.T. and C.M.U.P. (Portugal) through the programmes POCTI and POSI, with national and European
Community structural funds.}



\author[T. G\'omez]{Tomas L. G\'omez}

\address{Instituto de Ciencias Matem\'aticas (CSIC-UAM-UC3M-UCM),
Serrano 113bis, 28006 Madrid, Spain; and
Facultad de Ciencias Matem\'aticas,
Universidad Complutense de Madrid, 28040 Madrid, Spain}

\email{tomas.gomez@mat.csic.es}
\author[M. Logares]{Marina Logares}

\address{Departamento de Matematica Pura,
Facultade de Ciencias, Rua do Campo Alegre 687,
4169-007 Porto Portugal}

\email{mlogares@fc.up.pt}

\begin{document}

\begin{abstract}
In this article we extend the proof given by Biswas and G\'omez
\cite{BG} of a Torelli theorem for the moduli space of Higgs bundles
with fixed determinant, to the parabolic situation.
\end{abstract}

\maketitle

\section{Introduction}
The classical Torelli theorem says that if two smooth compact Riemann
surfaces have Jacobians which are isomorphic as polarized varieties,
then the Riemann surfaces are isomorphic.  In other words, a smooth
compact Riemann surface can be recovered from its principally
polarized Jacobian.  Analogously, in \cite{MN} Mumford and Newstead
proved that a smooth complex projective curve $X$ can be recovered
from the isomorphism class of $M^{2,\xi}_{X}$, the moduli space of
stable vector bundles over $X$ with rank $2$ and fixed determinant
$\xi$ of odd degree. This result was generalized by Narasimhan and
Ramanan in \cite{NR} where they extended it for any rank. They show
that the intermediate Jacobian of $M^{r,\xi}_{X}$ has a polarization
given by the positive generator of $\Pic(M^{r,\xi}_{X})$. This
polarized intermediate Jacobian is isomorphic to the Jacobian of the
curve $X$, so they reduce their proof to the classical Torelli
theorem.

Parabolic structures on holomorphic vector bundles were introduced by
Mehta and Seshadri in \cite{MS} inspired by the work of Weil. In
\cite{W} Weil studied the problem of characterizing fiber bundles
which arise from a representation of the fundamental group. The
parabolic situation corresponds to a representation of the fundamental
group of the complement of a finite set of marked points on a Riemann
surface with prescribed holonomy around the marked points. Mehta and
Seshadri constructed its moduli space using Geometric Invariant Theory
in \cite{MS}.

Higgs bundles appear when we consider non unitary representations of
the fundamental group, for instance representations into
$\GL(n,\CC)$. Hitchin started the study of those in \cite{H1} and gave
the moduli space the structure of a completely integrable Hamiltonian
system in \cite{H2}.

In \cite{BBB}, Balaji, del Ba\~no and Biswas prove a Torelli theorem
for parabolic bundles of rank $2$,
and in \cite{BG}, Biswas and G\'omez give a
Torelli theorem for Higgs bundles.  Our goal will be to provide the
moduli space of parabolic Higgs bundles with a Torelli type theorem.


In this article we prove the following (see
section \ref{sec:preliminaries} for notation):

\begin{theorem}
\label{maintheorem}
Let $X$ and $X'$ be Riemann surfaces of genus $g\ge 2$, let
$\SN_\xi(2,d,\alpha)$ and $\SN'_\xi(2,d,\alpha)$ be the moduli
spaces of parabolic Higgs bundles over $X$ and $X'$ respectively,
with fixed determinant $\xi$ of odd degree $d$. If
there is an isomorphism between $\SN_\xi(2,d,\alpha)$ and
$\SN'_\xi(2,d,\alpha)$ then there is an isomorphism between $X$ and
$X'$ sending the marked points at $X$ to the marked points at $X'$.
\end{theorem}

Note that we have to assume that the rank is $2$, because we
apply the Torelli theorem of \cite{BBB}.
Before proving the theorem, we need to do a detailed analysis of
the moduli space of parabolic Higgs bundles, and this is done
for any rank $r>1$ and degree $d$ coprime to $r$.

In \cite{BG}, it is shown that the moduli space of vector bundles,
which is embedded naturally inside the moduli space of Higgs bundles
as Higgs bundles with zero Higgs field,
can be characterized intrinsically. In other words, given the moduli
space of Higgs bundles as an abstract variety, we can find the
subvariety which corresponds to Higgs bundles with zero Higgs field,
and then we can apply
the Torelli theorem for the moduli space of vector
bundles.

Following the ideas of \cite{BG}, given the moduli space of parabolic
Higgs bundles as an abstract algebraic variety, we first recover the
Hitchin map. In section \ref{sec:origin} we use the Kodaira--Spencer
map and study the deformations of some objects related to the Hitchin
map to show that any $\CC^{\ast}$--action admitting a lift to the
moduli space of parabolic Higgs bundles,
has the origin as fixed point. Since the standard
$\CC^*$--action on the Hitchin space induced by sending a Higgs pair
$(E,\Phi)$ to $(E,t\Phi)$ has the origin as the unique fixed point, we
recover the origin of the Hitchin space.

By definition, the fiber over this point is the nilpotent cone,
and then we show that the
only irreducible component of the nilpotent cone which does not admit
a nontrivial $\CC^\ast$--action is the component corresponding to
parabolic bundles. Therefore, we have identified the moduli space of
parabolic bundles as a subvariety, and then we can apply \cite{BBB} to
recover the curve and the marked points.


\section{Preliminaries}
\label{sec:preliminaries}

Let $X$ be a smooth projective curve over $\CC$ of genus $g\ge 2$. Let
$D$ be a finite set of $n\geq 1$ distinct points of $X$. A
\emph{parabolic vector bundle} over $X$ is a holomorphic vector bundle
of rank $r$ (we assume $r\geq 2$)
together with a weighted flag on the fiber over each $x\in D$, called
\emph{parabolic structure}, that is
\begin{eqnarray*}
E&=E_{x,0}\supset E_{x,1}\supset\cdots\supset E_{x,r}\supset\{0\}\\
&0\le\alpha_{1}(x)< \cdots< \alpha_{r}(x)< 1 \, .
\end{eqnarray*}
The parabolic structure is said to have \emph{full flags} whenever
each step of the filtration has dimension one,
i.e. $\dim(E_{x,i}/E_{x,i+1})=1$. We denote
$\alpha=\{(\alpha_{1}(x),\ldots, \alpha_{r}(x))\}_{x\in D}$ to the
\emph{system of weights} corresponding to a fixed parabolic structure.

The \emph{parabolic degree} of a parabolic vector bundle is defined as
$$
\pdeg(E)=\deg(E)+ \sum _{i, x\in D} \alpha_{i}(x) \, ,
$$
and the \emph{parabolic slope} is then $\pmu(E)=\pdeg(E)/rk(E)$.

A parabolic bundle is said to be (semi)-stable if for all parabolic
subbundles, that is holomorphic subbundles $F\subset E$ with the
induced parabolic structure, the following condition for the parabolic
slope is satisfied
$$
\pmu(F)<\pmu(E) \qquad (\le) \, .
$$
The moduli space $\SM(r,d,\alpha)$ of semistable parabolic vector
bundles of rank $r$ and degree $d$ was constructed by Mehta and
Seshadri using Mumford's geometric invariant theory in \cite{MS}. They
also proved that $\SM(r,d,\alpha)$ is a normal projective variety of
dimension
$$
r^{2}(g-1)+1+ \frac{n(r^{2}-r)}{2} \, ,
$$
where the last summand comes from the fact that the flags we are
considering over each point of $D$ are full flags.
Moreover, it is smooth for a generic
choice of weights, where the system of weights $\alpha$ is
\emph{generic} whenever for such weights semistability implies
stability. From now on we assume the parabolic structure to have full
flags and generic weights.

Let $\xi$ be a line bundle on $X$. We denote $\SM_\xi$ the moduli
space of stable parabolic
vector bundles with fixed determinant. It is a projective scheme of dimension
$$
(g-1)(r^{2}-1)+\frac{n(r^{2}-r)}{2} \, .
$$

A \emph{parabolic Higgs bundle} $(E,\Phi)$ is a parabolic vector
bundle $E$ together with a homomorphism, called Higgs field,
$$
\Phi: E \too E\otimes K(D)
$$
such that it is a \emph{strongly parabolic} homomorphism, i.e. for
each point $x\in D$, the homomorphism induced in the fibers satisfies
\begin{eqnarray*}
\Phi(E_{x,i})\subset E_{x,i+1}\otimes K(D)|_x
\end{eqnarray*}
where $K$ is the canonical bundle over $X$, and $K(D)$ denotes the
line bundle $K\otimes \SO_{X}(D)$.

A parabolic Higgs bundle is called \emph{(semi)-stable} whenever the
slope condition holds for all $\Phi$-invariant parabolic subbundles,
that is $F\subset E$ parabolic subbundle, such that $\Phi(F)\subset
F\otimes K(D)$.

Denote $\SN(r,d,\alpha)$ the moduli space of stable rank $r$ degree
$d$ and weight type $\alpha$ parabolic Higgs bundles. It contains as
an open subset the cotangent bundle of the moduli space of stable
parabolic vector bundles and, under the assumption of genericity for
the weights, it is a smooth irreducible complex variety of dimension
$$
\dim \SN(r,d,\alpha)=r^{2}(2g-2)+2+n(r^{2}-r)
$$
where $n$ is the number of marked points on $X$.
Indeed, let $[E]\in \SM(r,d,\alpha)$. Denote by $\SPE(E)$ the
sheaf of strongly parabolic endomorphisms, i.e., endomorphisms
$\varphi:E\to E$ such that, for each point $x\in D$,
\begin{eqnarray*}
\varphi(E_{x,i})\subset E_{x,i+1} \; .
\end{eqnarray*}
Analogously, we say that an endomorphism is non-strongly parabolic
if it satisfies the weaker condition
\begin{eqnarray*}
\varphi(E_{x,i})\subset E_{x,i}
\end{eqnarray*}
and the sheaf of non-strongly parabolic endomorphisms is denoted
$\PE(E)$.
The tangent space at $[E]$ is isomorphic to
$H^{1}(X,\PE(E))$. By the parabolic version of Serre duality,
$$
H^{1}(X,\PE(E))^{\ast}\cong H^{0}(X,\SPE(E)\otimes K(D))
$$
and hence the Higgs field is an element of the cotangent space
$T^\ast_{[E]}\SM$, and the moduli space $\SM(r,d,\alpha)$ of stable
parabolic bundles is an open dense subset of the moduli
$\SN(r,d,\alpha)$ of parabolic Higgs bundles. The reader would like to
see \cite{Y,Y2,GGM} for references.

A trivial observation is that $\SM(r,d,\alpha)$ is embedded in
$\SN(r,d,\alpha)$, just take null Higgs fields.

Define the \emph{determinant map} from
$\SN(r,d,\alpha)$ to the cotangent
$T^{\ast}\Jac^{d}(X)$, which is canonically isomorphic to
$\Jac^{d}(X)\times H^0(X,K)$:
\begin{eqnarray*}
 \det:\SN(r,d,\alpha)&\to& \Jac^{d}(X)\times H^0(X,K),\\
(E,\Phi)&\mapsto& (\Lambda^r E, \tr \Phi) \, .
\end{eqnarray*}
This is well defined because, since $\Phi$ is strongly parabolic, its
trace actually lies in $H^0(X,K)\subset H^0(X,K(D))$.
Let $\xi$ be a
fixed line bundle of degree $d$. By definition, the fiber over $(\xi,
0)$ is called the moduli space of stable parabolic Higgs bundles with
fixed determinant $\xi$, i.e.
$$
\SN_{\xi}(r,d,\alpha):={\det}^{-1}(\xi,0) \, .
$$

Hence, for each $[(E,\Phi)]\in \SN_{\xi}(r,d,\alpha)$,  $\Phi$ is
a traceless $K$-valued meromorphic endomorphism, with simple poles in $D$,
and nilpotent residues respecting the parabolic filtration.

For fixed determinant $\xi$ it is
$$
\dim \SN_{\xi}(r,d,\alpha)=2 \dim \SM_{\xi}(r,d,\alpha)=
2(g-1)(r^{2}-1)+n(r^{2}-r) \, .
$$

\section{Hitchin map and spectral curves}\label{sec:hitchinmap}

We continue now defining the Hitchin map and the Hitchin space for
parabolic Higgs bundles. Let $S=\VV(K(D))$ be the total space of the
line bundle $K(D)$, let
$$
p:S=\underline{\Spec}\Sym^\bullet (K^{-1}\otimes \SO_{X}(D)^{-1}) \too X
$$
be the projection, and $x\in H^{0}(S,p^* (K(D)))$ be the
tautological section. The characteristic polynomial
of a Higgs field
$$
\det(x\cdot \id - p^* \Phi)=
x^r+\tilde{s_1} x^{r-1} + \tilde{s_2} x^{r-2} + \cdots + \tilde{s_r}
$$
defines sections $s_i\in H^0(X,K^iD^i)$, such that
$\tilde{s_{i}}=p^{\ast}s_{i}$ and $K^{i}D^{j}$ denotes the tensor
product of the $i$th power of $K$ with the $j$th power of the line
bundle associated to $D$.

We are assuming that
$\Phi$ is strongly parabolic,
therefore
the residue at each point of $D$ is nilpotent.
This implies
that the eigenvalues of $\Phi$ vanish
at $D$, i.e., for each $i>0$ the section $s_i$ belongs
to the subspace $H^{0}(X,K^i D^{i-1})$. We therefore define
the Hitchin space as
\begin{equation}
\SH = H^{0}(K) \oplus H^{0}(K^2D) \oplus \cdots \oplus H^0(K^rD^{r-1}).
\end{equation}

We also assume that $\Phi$ is traceless, i.e. $s_1=0$, and
then define the traceless Hitchin space as
\begin{equation}
\label{h0}
\SH_0 = H^0(K^2 D) \oplus \cdots \oplus H^0(K^r D^{r-1}) \, .
\end{equation}
Using Riemann--Roch, the parabolic Serre duality and the fact that
$\deg K^{n}D^{n-1}<0$ we obtain the dimensions of $\SH$ and
$\SH_{0}$. That is,
$$
\dim \SH_0=(g-1)(r^2-1)+\frac{n(r^2-r)}{2} \, ,
$$
which is equal to the dimension of $\SM_{\xi}$, and
$$
\dim \SH=(g-1)r^{2}+1+ \frac{n(r^{2}-r)}{2} \, ,
$$
which is equal to the dimension of $\SM(r,d,\alpha)$.

Taking the characteristic polynomial of a Higgs field
defines the Hitchin map
$$
h: \SN(r,d,\alpha) \too \SH
$$
which we can restrict to the moduli space of fixed determinant
$$
h_0: \SN_\xi(r,d,\alpha) \too \SH_0
$$
Given $s=(s_1,\ldots,s_r)\in \SH$, with
$$
s_i\in H^0(K^iD^{i-1}) \subset H^0(K^iD^i) \, ,
$$
we think of $s_i$ as a section of $K^iD^i$, and
then we define, as usual, the spectral
curve $X_s$ in $S$ as the zero scheme of the section of $p^*K^rD^r$
$$
\psi=x^r + \wt{s}_1 x^{r-1} + \wt{s}_2 x^{r-2} + \cdots + \wt{s}_r,
$$
where $\wt{s}_i=p^*s_i$. Denote by $\pi$ the restriction of the
projection $p$ to the spectral curve $X_s$
\begin{equation}
\label{eq:pi}
\pi: X_{s}=
\underline{\Spec}\big(\Sym^\bullet(K^{-1}D^{-1}/\SI)\big)\too X
\end{equation}
where $\SI$ is the ideal sheaf generated by the image of the homomorphism
$$
\begin{array}{rcl}
K^{-r}D^{-r}&\too &\Sym^\bullet(K^{-1}D^{-1})\\
\rule{0cm}{0.5cm}\alpha & \longmapsto & \alpha\sum_{i=0}^r s_i
\end{array}
$$
where we put $s_0=1$. From (\ref{eq:pi}) it follows that
\begin{equation}
\label{isopio}
\pi_*\SO_{X_s}=\SO_X  \oplus K^{-1}D^{-1}  \oplus K^{-2}D^{-2}
\cdots \oplus K^{-r+1}D^{-r+1}
\end{equation}

\begin{lemma}
\label{generic}
For $r\ge 2$ and $g\ge 2$ there is a dense open set in $\SH$ whose spectral
curve is smooth. The same holds for $\SH_0$.
\end{lemma}

\begin{proof}
  By remark $3.5$ in \cite{BNR} the set of sections $s$ such that the
scheme $X_s$ is smooth, is open, and it is nonempty whenever $K^r
D^{r-1}$ admits a section without multiple zeros.  It is known (IV
Corollary 3.2 \cite{H}) that, if $\deg(K^r D^{r-1})\ge 2g+1$,
then $K^r D^{r-1}$ is
very ample. It follows that, if $g\geq 2$, the line bundle
$K^r D^{r-1}$
has sections without multiple zeros.
\end{proof}

For $X_s$ smooth, the short exact sequences
$$
0 \too T_p \cong p^*K(D)  \too T_S \too p^* T_X \too 0,
$$
where $T_p$ denotes the relative tangent bundle for the
projection $p:S\too X$, and
$$
0 \too T_{X_s} \too T_S|_{X_s} \too N_{X_s/S} \too 0
$$
give
\begin{equation}
\label{iso0}
N_{X_s/S}\cong K_{X_s}\otimes \pi^* D \; .
\end{equation}
On the other hand,
\begin{equation}
\label{iso1}
N_{X_s/S}\cong \SO(X_s)|_{X_s}=p^* (K^r D^{r})|_{X_s}=\pi^* (K^r D^{r}),
\end{equation}
and the ramification line bundle of the projection $\pi:X_s\to X$
is
$$
\SO(R)=K_{X_s}\otimes\pi^* K^{-1}=\pi^*K^{r-1}D^{r-1}\; .
$$
The section
\begin{equation}\label{ramification}
\frac{\partial \psi}{\partial x}= r x^{r-1} + (r-1)\wt s_1 x^{r-2} +
\cdots + \wt s_{r-1} \;\in\; H^0(S,p^*K^{r-1}D^{r-1}),
\end{equation}
when restricted to $X_s$, gives a section of $\SO(R)$, and its
scheme of zeroes is exactly the ramification divisor $R$.

\begin{lemma}\label{lem:Prym-fibers}
If $X_s$ is smooth, then the fiber $h^{-1}(s)$ is
isomorphic to
$$
\op{Prym}(X_s/X)=\big\{L\in \Pic(X): \det\pi_*L\cong \xi\big\}
$$
\end{lemma}

\begin{proof}
Since $X_s$ is smooth, it is reduced, and hence
the vector bundle $E$ is of the
form $\pi_*L$ with $L$ an element of
$\op{Prym}(X_s/X)$. The morphism $\Phi$ is
given by multiplication by the tautological
section $x$, and it only remains to show that
the parabolic structure can be recovered from
$\Phi$.

For each parabolic point $p\in D$ there is a
Zariski open subset of $X$ where the vector bundle
is of the form
$E|_U=\SO_U[x]/I$ (seen as an $\SO_U$-module), where
$$
I=(x^r + s_{2}x^{r-2} + \cdots + s_r)
$$
Recall that
$s_i$ vanishes at $p$ for all $i$. Hence, the
fiber of $E$ over $p$ is
$$
E|_p = \CC[x]/(x^r).
$$
On the other hand $\Phi|_x$ is given
by multiplication by $x$, and hence it defines
a full flag on $E|_p$, which, by definition of
parabolic Higgs bundle, must coincide with the
parabolic filtration. Therefore, the parabolic
structure is recovered by $\Phi$.
\end{proof}
\begin{remark}
  Although this is written under the assumption of $\Phi$ being
  strongly parabolic it also works for non-strongly parabolic Higgs
  field. Obtaining, in such a case, the Jacobian of the spectral curve
  instead of the Prym variety (see \cite{LM}).
\end{remark}

\section{The nilpotent cone}\label{sec:nilpotentcone}

The fiber $h^{-1}(0)$ is called the nilpotent cone, and
it is a Lagrangian subscheme \cite[Thm. 3.14]{GGM}.
It follows that $h_{0}^{-1}(0)$ is also a Lagrangian
subscheme. Indeed, it is obviously isotropic, since it is
a subscheme of an isotropic variety, and the dimension
of its components is $\dim h^{-1}(0)-g$,
because it is the preimage in
$h^{-1}(0)$ of $(\xi, 0)\in T^{\ast}\Jac^{d}(x)$ under the determinant
map, and all the fibers of this map are isomorphic.
We call it the nilpotent cone for traceless fields.

Note that, if $(E,\Phi)$ belongs to the nilpotent cone, then
$\Phi$ is a nilpotent homomorphism (hence the name).
The moduli space of parabolic bundles is embedded in the nilpotent
cone as an irreducible component, via the map $E\mapsto (E,0)$.
In general there are other irreducible components corresponding
to nonvanishing nilpotent homomorphisms, and we are going to
see that the moduli space of parabolic bundles can be characterized
as the unique irreducible component of the nilpotent cone which
does not admit a non-trivial $\CC^*$--action. To show that it
does not admit a non-trivial $\CC^*$--action, it is enough to
show that $H^0(\SM_\xi,T_{\SM_\xi})=0$, since a non-trivial
$\CC^*$--action produces a non-zero vector field.

In \cite{NR} Narasimhan and Ramanan prove that if $\cS\SU_{\xi}$ is
the moduli space of stable vector bundles with fixed determinant, then
$H^0(\cS\SU_{\xi}, T_{\cS\SU_{\xi}})=0$.
We prove the following

\begin{proposition}\label{prop:novector}
Assume that the parabolic weights are generic (so that parabolic
semistable implies parabolic stable, and the moduli space is smooth)
and small enough so that the
stability of the parabolic Higgs bundle is equivalent to the stability
of the underlying vector bundle. Then,
$H^0(\SM_\xi,T_{\SM_{\xi}})=0$.
\end{proposition}

\begin{proof}
Let $f:\SM_\xi \to \cS\SU_\xi$ be the morphism sending a parabolic
bundle to the underlying vector bundle. Since we are assuming that the
weights are small enough, the stability of a parabolic
bundle coincides with the stability of the underlying vector
bundle. Hence this morphism is well defined and it is a projection
with fibers, fixing the weights, consisting of flag varieties giving
us the filtration over each marked point.

The exact sequence over $\SM_\xi$
$$
0\to T_f \to T_{\SM_\xi} \to f^\ast (T_{\cS\SU_\xi}) \to 0
$$
gives an isomorphism $H^0(\SM_\xi,T_{\SM_\xi})\cong H^0(\SM_\xi,
T_f)$, because $T_{\cS\SU_\xi}$ has no global
sections \cite[Theorem 1]{NR} and
$f$ is projective. The projection formula implies
$$
H^0(\SM_\xi, T_f)\cong H^0(\cS\SU_\xi, f_\ast (T_f)).
$$
We claim that $f_\ast(T_f)=End_{0}(\SE_x)$, where $\SE_x$ is the universal
vector bundle on $\cS\SU_\xi\times X$ restricted to the slice
$\cS\SU_\xi \times \{x\}$. Indeed, the fiber of $f$ over a point
corresponding to $E$ is the flag variety of the vector space $E_x$,
the fiber of $E$ over the point $x\in X$. This flag variety
is $\op{SL}(E_x)/B$ where $B$ is a Borel subgroup of $\op{SL}(E_x)$.
We have $H^0(\op{SL}(E_x)/B,T_{\op{SL}(E_x)/B})=\mathfrak{sl}(E_x)$
(\cite[Sec 4.8, Thm 2]{A}),
and therefore $f_\ast(T_f)=End_{0}(\SE_x)$.

Note that the universal vector bundle
exists because $r$ and $\deg(\xi)$ are coprime.
Therefore,
$$
H^0(\cS\SU_\xi, f_\ast (T_f))\cong
H^0(\cS\SU_\xi, End_{0}(\SE_x))=0\, ,
$$
where the last equality is in \cite[Theorem 2]{NR}.
\end{proof}

We can adapt Simpson's Lemma 11.9 in \cite{S} to the parabolic
situation with non zero degree, and then we obtain the following

\begin{lemma}\label{lem:comp}
Let $(E,\Phi)$ be a parabolic Higgs bundle in the nilpotent cone,
with $\Phi\ne 0$. Consider the standard $\CC^{\ast}$--action sending
$(E,\Phi)$ to $(E,t\Phi)$. Assume that $(E,\Phi)$ is a fixed
point. Then there is another Higgs bundle $(E',\Phi')$ in the
nilpotent cone, not isomorphic to $(E, \Phi)$ such that
$\limit_{t\to \infty}(E',t\Phi')=(E,\Phi)$.
\end{lemma}

These two results combine in the following

\begin{proposition}\label{prop:nilpotentcone}
Let $\SN_\xi(r,d,\alpha)$ be the moduli space of parabolic Higgs
bundles with fixed determinant $\xi$ over a compact Riemann surface
$X$ of genus $g\ge 2$, and let $h_0$ be the corresponding Hitchin
map. Then, there is only one component inside the nilpotent cone
$h_0^{-1}(0)$ which admits no nontrivial $\CC^{\ast}$--action,
and it is the moduli space of parabolic
bundles $\SM_{\xi}(r,d,\alpha)$
with fixed determinant $\xi$ over the compact Riemann
surface $X$.
\end{proposition}

\begin{proof}
Recall that the moduli space $\SM_{\xi}$ of parabolic bundles with
fixed determinant is embedded naturally in the moduli space
$\SN_{\xi}$ of parabolic Higgs bundles with fixed determinant just
by sending again $E\mapsto (E,0)$. From Proposition \ref{prop:novector} we
know that $\SM_{\xi}$ has no nontrivial $\CC^{\ast}$--action and
from dimensional computation it is one connected component of the
nilpotent cone. It remains to check that there is no other connected component
where there is no $\CC^{\ast}$--action.

In the rest of the components, with nonzero Higgs field, we have a
$\CC^{\ast}$--action given by $(E,\Phi)\mapsto (E,t\Phi)$.
And this action is nontrivial thanks to lemma \ref{lem:comp}.
\end{proof}

\section{The origin of the Hitchin space.}\label{sec:origin}
We are going to identify the nilpotent cone among the fibers of the Hitchin map.
For each point $s\in \SH$ of the Hitchin space we have a spectral
curve $X_s\subset S$ given by the spectral construction, and this
gives us a family of curves on $S$ parameterized by $\SH$.
Let $s\in \SH$ be a point such that the corresponding spectral curve
$X_s$ is smooth. If we move the point $s$, the spectral curve $X_s$
is deformed. The Kodaira--Spencer map
$$
u:T_s\SH\cong \SH\to H^1(X_s, T_{X_s})
$$
gives, for each vector in
the tangent space $T_s\SH$, the corresponding infinitesimal
deformation of the curve, which is described by an element of
$H^1(X_s,T_{X_s})$.
We will also consider the restriction of $u$ to the traceless Hitchin
space obtaining another Kodaira--Spencer map
$$
u_0:T_s\SH_0 \too H^1(X_s, T_{X_s}).
$$

A tangent vector at a point on the restricted Hitchin space $s\in
\SH_0$ defined by the standard action of $\CC^\ast$ is contained in
the kernel of the restricted Kodaira--Spencer map, because the
standard $\CC^*$--action does not change the isomorphism class
of the spectral curve. On the other hand, in Proposition
\ref{kernel} we show that the dimension of
the kernel of $u_0$ is one. Therefore, the direction defined
by the standard $\CC^*$--action coincides with the kernel of $u_0$.

By lemma \ref{lem:Prym-fibers}, the fiber of the Hitchin map $h_0$
over a point $s$ corresponding to a smooth spectral curve $X_s$ is
isomorphic to the Prim variety $P_s=\op{Prym}(X_s/X)$. Therefore,
if we move $s$ we will get a deformation of $P_s$. In particular,
we get a Kodaira--Spencer map between the infinitesimal deformations
$$
v_0:T_s\SH_0\too H^1(P_s,T_{P_s})
$$
The tangent vector defined by any $\CC^*$--action on $\SH_0$ which
lifts to $\SN_\xi(r,d,\alpha)$ is in the kernel of $v_0$. Indeed,
if the action lifts, it provides an isomorphism among the fibers
over the orbit of the action, hence the infinitesimal deformation
of $P_s$ has to be zero. We will see that a non-trivial
the curve $X_s$ produces a non-trivial deformation of $P_s$, and then we
obtain that the tangent vector defined by any $\CC^*$--action which
lifts to $\SN_\xi(r,d,\alpha)$ is in the kernel of $u_0$.


We use the following
\begin{lemma}
\label{lem:KD}
If $X_s$ is smooth, then there are natural isomorphisms
$$
T_s \wt\SH \;\cong\;H^0(X_s,\pi^*(K^{r}D^{r}))\;\cong\;H^0(X_s,N_{X_s/S}).
$$
\end{lemma}

\begin{proof}
Using the isomorphisms (\ref{iso1}) and (\ref{isopio}) and the projection
formula, we have
\begin{eqnarray*}
&H^0(X_s,N_{X_s/S}) = H^0(X_s,\pi^*(K^{r}D^{r})) =
H^0(X,K^{r} D^{r}\otimes \pi_*\SO_{X_s}) =&\\
&= H^0(X,\oplus_{i=1}^r (K^{i}D^{i})) = \wt\SH \cong T_s\wt\SH&
\end{eqnarray*}

\end{proof}

The objective now is to calculate the kernels
of $u$ and $u_0$. There are some elements in $H^0(X_s,N_{X_s/S})$ that are
clearly in the kernel. For instance, let
$\lambda\in H^0(X, \SO_X) \cong \CC $,
and denote a point in $X_s\subset S$ by $(\omega,v)$, where
$\omega$ is a point in $X$ and $x$  is a point
in the fiber of $S$ over $\omega$. Then the deformation sending
$(\omega,v)$ to $(\omega,e^\lambda v)$ clearly does not change the
isomorphism class of $X_s$. In fact, this is the deformation
produced by the standard $\CC^*$--action, and it is clearly
in the kernel of the Kodaira--Spencer map $u_0$.

Furthermore, for any $\alpha\in H^0(X,K(D))$, sending $(\omega,
v)$ to $(\omega,v+\alpha(\omega))$
also preserves the isomorphism class of $X_s$. The deformations
defined in this way do not preserve the condition
$0=\tr(\Phi)$ $(=s_1)$, and hence they are in the kernel of $u$,
but not in $T_s\SH_0$. The following proposition
says that these two constructions describe the kernels.

\begin{proposition}
\label{kernel}
The kernel of the Kodaira--Spencer map $\wt u$ is given by the following
exact sequence
$$
0 \too H^0(X,K(D) \oplus \SO_X) \too T_s \wt\SH \stackrel{\wt u}
\too H^1(X_s,T_{X_s})
$$
so it has dimension $g+n+1$ (recall $n=\deg D$).
If we fix the determinant, the restriction of the Kodaira--Spencer map provides an exact sequence
$$
0 \too H^0(X,\SO_X) \too  T_s\wt\SH_0
\stackrel{\wt u_0}{\too} H^1(X_s,T_{X_s}),
$$
and hence $\dim \Ker \wt u_0 =1$.
If we restrict the Kodaira--Spencer
map to $\SH_0$ then we have
$$
0 \too H^0(X,\SO_X) \too  T_s\SH_0
\stackrel{u_0}{\too} H^1(X_s,T_{X_s}),
$$
and hence $\dim \Ker \wt u_0 =1$.

\end{proposition}

\begin{proof}
Consider the following diagram, constructed using (\ref{iso1})
\begin{equation}
\label{pd}
\xymatrix{
  & & {p^*T_{X}|_{X_s}}  & & \\
{0} \ar[r] &  {T_{X_s}} \ar[r] &
{T_S|_{X_s}} \ar[rr]\ar@{>>}[u] && {N_{X_s/S}}
\ar[r] & 0 \\
 & & {\rule{20pt}{0pt}T_p|_{X_s}\cong \pi^*K(D)} \ar@{^{(}->}[u]
\ar[rr]^-{\otimes\frac{\partial \psi}{\partial x}\big|^{}_{X_s}}
&& {\pi^*(K^rD^{r})} \ar@{=}[u] &
}
\end{equation}
where $T_p$ denotes the relative tangent bundle for the
projection $p:S\too X$, as in section \ref{sec:hitchinmap}.
Note that the diagram is well defined, since
$\frac{\partial \psi}{\partial x}\big|_{X_s}$ is a section of
$$
\pi^*(K^{r-1}D^{r-1})\cong \SO(R)
$$
(cf.(\ref{ramification})).
The diagram is commutative because the zero scheme
of the two morphisms between the line bundles $T_p|_{X_s}$ and
$N_{X_s/S}$
are the same, namely the ramification divisor, hence the maps
differ by a scalar, but this scalar
can be absorbed in the
projection $T_S|_{X_s}\to N_{X_s/S}$.

Since the tangent line bundle $T_X$ has negative degree,
$H^0(X_s, p^*T_{X}|_{X_s})=0$. Therefore, the middle column
in diagram (\ref{pd}) gives
$$
H^0(X_s, T_S|_{X_s}) \cong H^0(X_s, T_p|_{X_s}) \; .
$$
On the other hand,
$$
H^0(X_s, T_p|_{X_s}) \cong H^0(X_s,\pi^* K(D)) \cong
$$
$$
\cong H^0(X,K(D)\otimes \pi_* \SO_{X_s}) \cong
H^0(X,K(D)\oplus \SO_X)
$$
which, together with lemma \ref{lem:KD},
transforms the long exact sequence given by the middle row
of diagram (\ref{pd}),
$$
0\to H^0(X_s,T_{X_s})\to H^0(X_s,T_S|_{X_s})\to H^0(X_s, N_{X_s/S})\to H^1(X_s,T_{X_s}) \; ,
$$
into
$$
0 \too H^0(X,K(D) \oplus \SO_X) \too T_s\wt\SH
\stackrel{u}{\too} H^1(X_s,T_{X_s}),
$$
where we have used $T_s\wt H\cong H^0(X_s,\pi^*(K^r D^{r}))$.

Now we restrict the Kodaira--Spencer map $u$ to $T_s\wt\SH_0$.
Using the isomorphism $H^0(X_s,\pi^*(K^r D^{r}))\cong
H^0(X,\oplus_{i=1}^r (K^iD^{i}))$ (cf. proof of Lemma \ref{lem:KD}),
an element of this group is written as
$$
\wt a_0 x^r + \wt a_1 x^{r-1} + \cdots + \wt a_r
$$
with $a_i\in H^0(X, K^iD^{i})$, $a_0=1$ and $\wt a_i=\pi^* a_i$.

On the other hand,
an element of $H^{0}(X,K(D)\oplus \SO_X)$ can be written as
$$
\wt b_1 + \wt b_0 x
$$
with $b_i\in H^0(X, K^iD^{i})$ and $\wt b_i=\pi^* b_i$.

The map $\theta$ from one to another is $H^0(\otimes\frac{\partial \psi}{\partial x}\big|_{X_s})$. A short calculation
(using $\psi|_{X_s}=0$) gives
\begin{equation}\label{exfor}
\wt\theta(\wt{b}_1+\wt{b}_0x)=
\sum_{i=1}^{r}
\big( (r-i+1) \wt s_{i-1}\wt b_1 - i \wt s_i \wt b_0\big)x^{r-i}
\end{equation}

The subspace $\wt\SH_0\subset \wt\SH$ is the zero locus of the
trace map sending $(s_1,\ldots,s_r)$ to $s_1$. Then we
have a commutative diagram
$$
\xymatrix{
0 \ar[r] & {T_s\wt\SH_0} \ar[r] & {T_s\wt\SH} \ar[r]^{d(\tr)} &
H^0(X,K(D)) \ar[r] & 0 \\
0 \ar[r] & {H^0(\oplus_{i=2}^n (K^iD^{i}))} \ar[r]\ar[u]^{\cong} &
{H^0(\oplus_{i=1}^n (K^iD^{i}))} \ar[r]^{\qquad p_1} \ar[u]^{\cong} &
H^0(X,K(D)) \ar[r] \ar@{=}[u] & 0
}
$$
where $p_1$ is projection to the first summand.

Now, if $s\in \wt\SH_0$, then $s_1=0$, and using the explicit
formula (\ref{exfor}), we obtain
$$
(d(\tr) \circ \wt\theta) (\wt b_1 + \wt b_0 x) \;=\; r\wt b_1,
$$
and hence the following diagram is commutative
$$
\xymatrix{
& 0  & 0   \\
 & {H^0(K(D))} \ar[u]\ar@{=}[r] & {H^0(K(D))}\ar[u] &  \\
0 \ar[r] & {H^0(K(D) \oplus \SO_X)} \ar[r]^-{\wt\theta} \ar[u]^{q} &
{T_s\wt\SH} \ar[r]^-{\wt u}\ar[u]^{d(\tr)} & {H^1(X_s,T_{X_s})}  \\
0 \ar[r] & {H^0(\SO_X)} \ar[r]^-{\wt\theta_0} \ar[u] & {T_s\wt\SH_0} \ar[r]^-{\wt u_0}\ar[u]
& H^1(X_s,T_{X_s})  \ar@{=}[u]  \\
 & 0\ar[u] & 0\ar[u] & \\
}
$$
where $q$ is projection to the first summand followed by
multiplication by $r$. The top row is the
second exact sequence in the statement of the proposition.

Finally, if the spectral curve corresponds to a strongly parabolic homomorphism, i.e. $s=(s_1,\ldots,s_r)\in \SH_0$ with
$$
s_i\in H^0(K^iD^{i-1}) \subset H^0(K^iD^i) \; ,
$$
then the image of $\wt\theta_0$ lies in
$$
H^0(X,\oplus_{i=2}^r(K^iD^{i-1}))\cong T_s\SH_0
$$
and therefore we get a commutative diagram
$$
\xymatrix{
0 \ar[r] & {H^0(\SO_X)} \ar[r]^-{\wt\theta_0}  & {T_s\wt\SH_0} \ar[r]^-{\wt u_0}
& H^1(X_s,T_{X_s})    \\
0 \ar[r] & {H^0(\SO_X)} \ar[r]^-{\theta_0} \ar@{=}[u] & {T_s\SH_0} \ar[r]^-{u_0}\ar[u]
& H^1(X_s,T_{X_s})  \ar@{=}[u]  \\
}
$$
\end{proof}

\begin{proposition}\label{prop:origin}
  Let $g:\CC^{\ast}\times \SH_0 \too \SH_0$ be an action, having
  exactly one fixed point, and admitting a lift to $\SN_\xi
  (r,d,\alpha)$. Then this fixed point is the origin of $\SH_0$.
\end{proposition}

\begin{proof}
The proof is the same as in \cite[Proposition 5.1]{BG}, so we
only give a sketch, since the details can be found there.

Let $s\in \SH_0$ be a point corresponding to a smooth spectral curve
$X_s$ (it exists by Lemma \ref{generic}).
The tangent vector defined at this point by the standard
action is contained in the kernel of the Kodaira--Spencer map $u_0$,
since the standard action does not change the isomorphism class of the
spectral curve. We are going to prove that the tangent vector
defined by any action that lifts to $\SN_\xi(r,d,\alpha)$ is also in
the kernel of the Kodaira--Spencer map.

Let $J=\textup{Jac}(X)$, $J_s=\textup{Jac}(X_s)$, and
$P_s=\textup{Prym}(X_s/X)$. There is an \'etale covering
$$
\alpha: P_s\times J \to J_s
$$
sending $(L_1,L_2)$ to $L_1\otimes \pi^*L_2$.

Let $g:\CC^*\times \SH_0\to \SH_0$ be an action. Its derivative
gives a tangent vector $w$ at $s$, and the image of $w$
under the Kodaira--Spencer map produces an infinitesimal
deformation of the spectral curve $u_0(w)=\eta_1\in H^1(X_s,T_{X_s})$,
its Jacobian $\eta_2\in H^1(J_s,T_{J_s})$, and the Prym variety
$\eta_3\in H^0(P_s,T_{P_s})$. We have homomorphisms
$$
\xymatrix{
{H^1(X_s,T_{X_s})} \ar@{^{(}->}[r]^-{i} &
{H^1(J_s,T_{J_s})} \ar@{^{(}->}[r]^-{\epsilon}  &
{H^1(J\times P_s,T_{J\times P_s})}\\
& & {H^1(P_s,T_{P_s})} \ar@{^{(}->}[u]\\
}
$$
Indeed, a deformation of $X_s$ produces a deformation of $J_s$, and
the corresponding homomorphism $i$ is injective because of the
infinitesimal version of the classical Torelli theorem for curves
(a non-zero deformation of a curve produces a non-zero deformation of
its Jacobian). On the other hand, an infinitesimal deformation of $J_s$ produces an infinitesimal deformation of its \'etale covering $J\times P_s$. The image of the composition $\epsilon\circ i$
lies in $H^1(P_s,T_{P_s})$, because a deformation of $J\times P_s$ induced by a deformation of $X_s$ (keeping $X$ constant) is induced
by a deformation of the factor $P_s$.

We have $i(\eta_1)=\eta_2$ and $\epsilon(\eta_2)=\eta_3$.
Recall that the fiber of $h_0$ over $s$ is canonically isomorphic
to $P_s$. Now, if the action $g$ lifts to $\SN_\xi(r,d,\alpha)$,
all the fibers above the points of an orbit in $\SH_0$ should
be isomorphic, and therefore we should have $\eta_3=0$.
By the injectivity of $\epsilon$ and $i$, this implies that
$\eta_1=0$, i.e., $u_0(w)=0$, and then the tangent vector
$w$ defined by the action $g$ is in the kernel of the Kodaira--Spencer map.

Now let $g$ be an action of $\CC^*$ on $\SH_0$, which has exactly one
fixed point, and which admits a lift to
$\SN_\xi(r,d,\alpha)$. There is a dense open set $U$ of $\SH_0$
corresponding to smooth spectral curves (Lemma \ref{generic}).
In this open set, each
orbit of $g$ is included in an orbit of the standard action,
since for each point in this open set, the span of the tangent
vector defined by the standard action coincides with the kernel of the
Kodaira--Spencer map, and the tangent vector defined by $g$ is in
this kernel. In particular, the origin, being the fixed point of
the standard action, is a limiting point of all the orbits of
$g$ in $U$, but it is not in the orbits of $g$ (because the fiber
over the origin is not isomorphic to the fibers over $U$). The
limiting points of a $\CC^*$--action are fixed points, therefore
the origin is a fixed point of $g$, and, by hypothesis is the only
fixed point.
\end{proof}

\section{Proof of main theorem}\label{sec:main}

In this section we prove Theorem \ref{maintheorem}.
Since we are going to apply the Torelli theorem for parabolic
bundles in \cite{BBB}, we have to assume that the rank is 2.

\begin{proof}[Proof of Theorem \ref{maintheorem}.]

The Hitchin map gives a surjective morphism from
$\SN_\xi(2,d,\alpha)$ to the vector space $\SH$ called the
\emph{Hitchin space}. The fiber over the origin is called the
nilpotent cone. Observe that $\SM_\xi(2,d,\alpha) \subset
h_0^{-1}(0)$ and has the same dimension as the fiber so, it is one
connected component of this fiber. Actually, we will show that it is
the unique irreducible component inside the nilpotent cone that does
not admit a non-trivial $\CC^{\ast}$--action. Therefore, using the
Torelli theorem in \cite{BBB}, we recover the pointed curve.
This is the idea used in \cite{BG} to
determine intrinsically $\SM_\xi(2,d,\alpha)$ inside
$\SN_\xi(2,d,\alpha)$. We show that it also works in the parabolic
situation.

Note that we are given only the isomorphism class of
$\SN_\xi(2,d,\alpha)$ so we do not have explicitly the Hitchin map.

Consider $Y$ an algebraic variety isomorphic to our moduli space
$\SN_\xi(2,d,\alpha)$ and the natural morphism $Y\to \Spec (\Gamma
(Y))$. Since $Y$ is isomorphic to $\SN_\xi(2,d,\alpha)$ and the fibers
of the Hitchin map are compact by Lemma \ref{lem:Prym-fibers} it
happens that the ring of global functions of $\SN_\xi(2,d,\alpha)$
factorizes through the Hitchin map and
$\Spec(\Gamma(Y))\cong\Spec(\Gamma(\SH_0))$ and hence
$$
\Spec (\Gamma(Y) ) \cong\Spec(\Gamma(\SH_0))\cong
\CC[y_1,y_2,\ldots,y_{3(g-1)+n}].
$$

The following diagram commutes,
$$
\xymatrix{
Y\ar[r]^\alpha \ar[d]^m &\SN_\xi(2,d,\alpha)\ar[d]^{h_0}\\
\mathbb{A}^{3(g-1)+n}\ar[r]_\beta &\SH_0
 }
$$
Let $g:\CC^{\ast}\times \mathbb{A}^{3(g-1)+n} \too
\mathbb{A}^{3(g-1)+n}$ be a $\CC^{\ast}$--action with exactly
one fixed point $y$ and such that it admits a lift to $Y$. We know
that it exists: take, for instance, the standard $\CC^{\ast}$--action on
$\SH_0$ and apply it to $\mathbb{A}^{3(g-1)+n}$ through
the isomorphism $\beta$.

We also know that such an action has as unique fixed point the origin
by Proposition \ref{prop:origin}. Hence, the fiber over $y$ is
isomorphic to the nilpotent cone $h_0^{-1}(0)$. It only remains to use
Proposition \ref{prop:nilpotentcone} to identify $\SM_\xi(2,d,\alpha)$ as
the component inside
the nilpotent cone which does not admit a nontrivial
$\CC^{\ast}$--action. By Theorem 3.2 in \cite{BBB}, from the
isomorphism class of $\SM_\xi(2,d,\alpha)$
we recover $(X,x)$ up to isomorphism.
\end{proof}

\end{document}